\documentclass[11pt]{amsart}

\usepackage[utf8x]{inputenc} 
\usepackage[a4paper]{geometry} 
\usepackage{enumerate}  
\usepackage{amsmath} 
\usepackage{amssymb}
\usepackage{color}
\usepackage{multirow}
\usepackage{hyperref}  
\usepackage{verbatim}
\usepackage[normalem]{ulem}
\usepackage{comment}
                 \hypersetup{ pdfborder={0 0 0}, 
                              colorlinks=true, 
                              linktoc=page,
                              pdfauthor={Michael Hoff and Isabel Stenger}, 
                              pdftitle={On the numerical dimension of Calabi-Yau 3-folds of Picard number 2}} 
\definecolor{Gray}{gray}{0.9}                            
\usepackage[color, arrow,curve,matrix,tips,frame]{xy}
\usepackage{graphicx}

\theoremstyle{plain} 
\newtheorem{proposition}{Proposition}[section] 
\newtheorem{theorem}[proposition]{Theorem} 
\newtheorem{lemma}[proposition]{Lemma} 
\newtheorem{corollary}[proposition]{Corollary} 
\newtheorem*{maintheorem}{Main Theorem} 
\theoremstyle{definition} 
\newtheorem{definition}[proposition]{Definition} 
\newtheorem{notation}[proposition]{Assumption} 
\newtheorem{example}[proposition]{Example} 

\theoremstyle{remark} 
\newtheorem{remark}[proposition]{Remark}

\renewcommand{\O}{{\mathcal{O}}}

\newcommand{\PP}{{\mathbb{P}}}          

\newcommand{\cA}{{\mathcal{A}}} 
\newcommand{\cB}{{\mathcal{B}}}

\DeclareMathOperator{\Amp}{Amp}
\DeclareMathOperator{\Nef}{Nef}
\DeclareMathOperator{\Eff}{Eff}
\DeclareMathOperator{\Bir}{Bir}
\DeclareMathOperator{\BigCone}{Big}
\DeclareMathOperator{\Aut}{Aut}
\DeclareMathOperator{\Int}{int}
\DeclareMathOperator{\id}{id}
\DeclareMathOperator{\MovC}{\overline{Mov}}
\DeclareMathOperator{\EffC}{\overline{Eff}}
\DeclareMathOperator{\Mov}{Mov}

\numberwithin{equation}{section}

\title[On the numerical dimension of Calabi-Yau 3-folds of Picard number 2]{On the numerical dimension of Calabi-Yau 3-folds of Picard number 2}

\author[M. Hoff]{Michael Hoff} 
\author[I. Stenger]{Isabel Stenger} 
\address{Universit\"at des Saarlandes, Campus E2 4, D-66123 Saarbr\"ucken, Germany}
\email{\href{mailto:hahn@math.uni-sb.de}{hahn@math.uni-sb.de}} 
\email{\href{mailto:stenger@math.uni-sb.de}{stenger@math.uni-sb.de}} 

\date{\today} 

\begin{document}

\begin{abstract} 
We show that for any smooth Calabi-Yau threefold $X$ of Picard number $2$ with infinite birational automorphism group, the numerical dimension $\kappa_\sigma$ of the extremal rays of the movable cone of $X$ is $\frac{3}{2}$.  Furthermore, we provide new examples of Calabi-Yau threefolds of Picard number $2$ with infinite birational automorphism group. 
\end{abstract}

\maketitle

\section{Introduction} 

Given a normal variety $X$ and a  divisor $D$, the asymptotic growth rate of $h^0(X,mD)$ for increasing $m$ plays a fundamental role in birational geometry. Unfortunately, the  Iitaka dimension measuring this growth rate is not invariant under the numerical equivalence relation. To overcome this problem, there are several definitions of a numerical invariant,  the numerical dimension of $D$, which were commonly believed to be equivalent (see \cite{Nak04, Leh13, Eckl16, Fuj20, CP21}). 
However, in \cite{L21}, Lesieutre shows that the different notions of numerical dimension for a pseudoeffective $\mathbb{R}$-divisor do not coincide. In particular, he shows the existence of a pseudoeffective $\mathbb{R}$-divisor whose numerical dimension $\kappa_\sigma$ is not an integer.  
\begin{definition}{(\cite[\S 5]{Nak04})}
	The numerical dimension $\kappa_\sigma(X,D)$ is the largest (real) number $k$ such that for some ample divisor $A$,  one has 
	$$
	\limsup_{m\to \infty} \frac{h^0(X,\lfloor mD\rfloor + A)}{m^k} > 0.
	$$
	If no such $k$ exists, we take $\kappa_\sigma(X,D) = -\infty$.
\end{definition}

\begin{remark}
	We refer to \cite{L21} and the references therein for properties of $\kappa_\sigma$ and different definitions of the numerical dimension.
\end{remark}
In a specific example (first studied in \cite[Section 6]{Og14}),  Lesieutre studies the growth of a boundary divisor of the movable cone of a (general) complete intersection of type $(1,1),(1,1)$ and $(2,2)$ on $\PP^3 \times \PP^3$,  which is a Calabi-Yau threefold of Picard number $2$ whose birational automorphism group is not finite.  The main result of this paper is that Lesieutre's example is not a sporadic one but holds in a more general setting.  
\begin{maintheorem}\label{maintheorem}
	 Let $X$ be a smooth Calabi-Yau threefold with Picard number $\rho(X) = 2$ such that its group of birational automorphisms $\Bir(X)$ is infinite. Let $D$ be an $\mathbb{R}$-divisor on the boundary of the movable cone $\MovC(X)$. Then $\kappa_\sigma(X,D) = \frac{3}{2}$. 
\end{maintheorem}

The original goal of this paper was to generalize Lesieutre's idea to $\mathbb{Q}$-divisors. This cannot be achieved for a Calabi-Yau threefold $X$ as in the main theorem since there are no $\mathbb{Q}$-divisors on the boundary of the movable cone of $X$ (see Remark \ref{remarkOguiso}). Therefore, the question whether the different notions of numerical dimensions are equivalent for pseudoeffective $\mathbb{Q}$-divisors is still open.  Even the recent developments of \cite{CP21} do not answer the question if the problem of the definition of the numerical dimension by Nakayama is in the extension of the definition to $\mathbb{R}$-divisors or in the nature of the definition itself. 

The outline of the proof of our main theorem is as follows. We describe the movable cone of $X$ as well as the fundamental domain for the action of the birational automorphism group on it. This relies on \cite{LP13} which shows that the Cone conjecture is true for Calabi-Yau varieties of Picard number $2$.  Then we compute the numerical dimension of a boundary divisor of the movable cone of $X$ following \cite{L21}.  The idea is to pullback a divisor close to the boundary to the fundamental domain (via a birational automorphism). Then we compute the dimension of the space of global sections for any divisor $D$ in the fundamental domain  on another minimal model of $X$ where the strict transform of $D$ is nef (see Corollary \ref{coverNef}). To perform the computation of the numerical dimension, we use some results of \cite{Kaw88, Kaw97} for Calabi-Yau threefolds. To the end we show that the procedure is stable under round-down for $\mathbb{R}$-divisors. 

We cannot generalize our strategy of the proof of the main theorem to higher Picard number $\rho(X)$ since the Cone conjecture is not known to be true. It is essential in our proof that we can pullback a divisor to a fixed cone in the movable cone of $X$. If $\Bir(X)$ is finite the presented method does not apply. 

Section \ref{examples} is devoted to the construction of new Calabi-Yau threefolds $X$ of Picard number $2$ with infinite birational automorphism group. We construct a fivefold $Z$ of Picard number $2$ isomorphic to the projectivisation of the universal rank $2$ quotient bundle on the Grassmannian $G(2,4)\subset \PP^5$. Our constructed Calabi-Yau threefolds are complete intersections on $Z$.  

The case of Calabi-Yau varieties with infinite birational automorphism group and Picard number greater than $2$ will be treated in a forthcoming paper.

\subsection*{Acknowledgments} We would like to thank Vladimir Lazi\'c for proposing this topic and helpful discussions. We thank Frank-Olaf Schreyer for providing useful feedback on the construction of our examples. 
We thank John Lesieutre for sharing with us the final version of his article.  We also thank the referee for careful reading and detailed suggestions. 
Both authors were supported by the Deutsche Forschungsgemeinschaft (DFG, German Research Foundation) - Project-ID 286237555 - TRR 195.   

\section{Preliminaries and notation}

We follow \cite{LP13} to introduce the basic objects and notation of the article.

\begin{notation}\label{mainassumption}
 Unless otherwise stated, $X$ is a smooth Calabi-Yau threefold (that is, a projective variety of dimension $3$ with trivial canonical bundle and $H^1(X,\O_X) = 0$) with Picard number $\rho(X) = 2$ such that its group of birational automorphisms $\Bir(X)$ is infinite.
\end{notation}

\subsection*{The Cone conjecture for the movable cone}\label{movX}

Let $N^1(X)$ be the N\'eron–Severi group.  We denote by $\Nef(X) \subset N^1(X)_{\mathbb{R}}$ the nef cone,  by $\MovC(X)\subset N^1(X)_{\mathbb{R}}$ the closure of the cone generated by movable  divisors on $X$, and by $\Mov(X)$ its interior. Since $\rho(X)=2$, let $m_1,m_2$ be the two boundary rays of $\MovC(X)$.  Furthermore, let $\Eff(X)$ be the effective cone with closure $\EffC(X)$, the pseudoeffective cone. Finally, we denote by $\cB(X)$ (resp. $\cA(X)$) the group acting on $\MovC(X)$ (resp. $\Nef(X)$) induced by $\Bir(X)$ (resp. $\Aut(X)$). 

\begin{remark}{\cite[Proposition 3.1 and Corollary 2.6]{Og14}}\label{remarkOguiso}
 If one of the extremal rays $m_1$ and $m_2$ is rational, then $\Bir(X)$ is finite. In particular, if $\Bir(X)$ is infinite, then the both rays $m_1$ and $m_2$ are irrational. 
We note that there is no example of a Calabi-Yau variety with Picard number $2$ and finite birational automorphism group such that (at least) one of the boundary divisors of the movable cone is irrational. 
\end{remark}

One of the main results of \cite{LP13} is that the Cone conjecture for $\MovC(X)$ is true. 
\begin{theorem}[{\cite[Thm.\ 4.5]{LP13}}]
Let $X$ be a smooth Calabi-Yau variety with Picard number $2$ and infinite $\Bir(X)$.  There exists a rational polyhedral cone $\Pi$ which is a fundamental domain for the induced action of $\Bir(X)$ on $\MovC(X)\cap \Eff(X)$,  that is, 
$$
\MovC(X)\cap \Eff(X) = \bigcup_{g\in \cB(X)} g\Pi
$$
with $\Int\Pi \cap \Int g \Pi = \emptyset$ unless $g=\id$.
\end{theorem}

 We recall how the fundamental domain is constructed depending on the structure of $\cB(X)$. If $g$ is any element of $\cB(X)$, then $\det g = \pm 1$ since $g$ acts on the integral lattice $N^1(X)$.  Let 
$$
\cB^\pm(X) = \{g \in \cB(X)\ | \ \det g = \pm 1\}.
$$
The sets $\cA^\pm(X)$ are defined analogously. We collect some results of \cite{LP13} which are important for us.

\begin{theorem}[{\cite[Thm.\ 3.9, Thm.\ 4.7, Lem.\ 4.4]{LP13}}]\label{propertiesCY}
	Let $X$ be a smooth Calabi-Yau variety with Picard number $2$ and infinite $\Bir(X)$.  Then:
 \begin{itemize}
 \item $\cB^+(X) \cong \mathbb{Z}$.
 \item Let $\sigma\in \cB^+(X)$ be a generator, then the boundary rays of $\Mov(X)$ are spanned by eigenvectors of $\sigma$.  In particular, $\sigma$ is diagonalizable and the eigenvalues are different from one. 
 \item There exists a birational automorphism $\tau\in \cB^-(X)$ such that $\tau^2\in \Aut(X)$, and a birational automorphism of infinite order $\sigma\in \cB^+(X)$ such that 
 $$
  \cB(X) = \langle \cA(X), \sigma, \tau \rangle. 
 $$
 The birational automorphism $\tau$ is possibly the identity, depending if $\cB^-(X)= \emptyset$ or not.
 \item $|\cA(X)|\le 2$, $|\cA^+(X)| =1$ and $\cA^-(X) = \cA^+(X)h$ for any  $h \in \cA^-(X)$. This means that an automorphism of $X$ acts either as the identity on the movable cone or permutes the boundary rays of the nef cone $\Nef(X)$.
 
 \item $\MovC(X) = \EffC(X)$ and $\Mov(X) = \MovC(X) \cap \Eff(X)$.
 \end{itemize}
\end{theorem}

\begin{example}[{\cite[Section 6]{Og14}}]
 A general complete intersection $X$ of type $(1,1)$, $(1,1)$ and $(2,2)$ on $\PP^3 \times \PP^3$ is a Calabi-Yau threefold of Picard number $2$ with infinite birational automorphism group $\Bir(X)$. Furthermore, $\cB^-(X)\neq \emptyset$, the fundamental domain for the action of $\Bir (X)$ on $\MovC(X)$ is the nef cone $\Nef(X)$, and both boundary rays of $\Nef(X)$ are fixed by some birational automorphism in $\cB^-(X)$. 
\end{example}

 This example motivates the following proposition which shows that for a Calabi-Yau threefold as in Assumption \ref{mainassumption}, the properties of the boundary rays of the fundamental domain are the same when $\cB^-(X)\neq \emptyset$.  We assume in the following proposition that the automorphism group acts trivially on the movable cone.

\begin{proposition}\label{prop_NefInFundamental}
Let $X$ be a smooth Calabi-Yau variety with Picard number $2$ and infinite $\Bir(X)$ such that $\cA(X) = \id$. Then:
\begin{enumerate}
 \item [(a)] There exists a fundamental domain $\Pi$ for the action of $\Bir(X)$ on $\MovC(X)$ such that $\Nef(X) \subseteq \Pi$.
 \item [(b)] If  $\cB^-(X)\neq \emptyset$, there exist birational automorphisms $\tau_1$ and $\tau_2$ in $\cB^-(X)$ each fixing a boundary ray of $\Pi$.
\end{enumerate}
\end{proposition}

\begin{proof}
First, we recapitulate the construction of the fundamental domain $\Pi$ given in \cite[Theorem 4.5]{LP13} and show afterwards that we can always move $\Pi$ in $\MovC(X)$ around, to obtain $\Nef(X) \subseteq \Pi$. 
	
We start with the case $\cB^{-}(X)= \emptyset$. Let $\sigma$ be  a generator of $\cB(X) \cong \cB^+(X) \cong \mathbb{Z}$. By {\cite[Theorem 2.4]{LP13}}, for any $y \in \Amp(X)$, we have $\sigma y \notin \Amp(X)$ as $\sigma$ is not an automorphism of $X$. We show now that for any divisor $x$ on the boundary of $\Nef(X)$, the divisors $\sigma x$ and $\sigma^{-1}x$ are not contained in $\Amp(X)$. Let $(x_k)_k \in \Amp(X)$ a sequence converging to $x$ on the boundary of $\Nef(X)$. Then the sequence $(\sigma x_k)_k$ converges to $\sigma x$ and the claims follows as $\sigma x_k \notin \Amp(X)$ for each $k$. The same holds for $\sigma ^{-1}$.
Thus, after replacing $\sigma$ by $\sigma^{-1}$ if necessary, we may assume that for an element $x$ in one of the two boundary rays of $\Nef(X)$, the cone   
$$
\Pi = \mathbb{R}_{\geq 0}x + \mathbb{R}_{\geq 0}\sigma x,
$$ 
contains $\Nef(X)$. Furthermore,  $\Pi$ and $\sigma\Pi$ have a non-overlapping interior and  

$$ \text{Mov}(X) = \MovC(X) \cap \text{Eff}(X) = \bigcup_{n \in \mathbb{Z}} \sigma^n \Pi$$ as
$\sigma^nx$ converges (after scaling) to the boundary rays $m_1$ and $m_2$ of $\MovC(X)$ for $n \rightarrow \pm \infty$.  This follows from the fact that the boundary rays are spanned by eigenvectors of $\sigma$ (see \cite{Bir67} and also \cite[Lemma 2]{L21}). 

Next let us assume that $ \cB^-(X) \neq \emptyset$.  Let $\sigma$ be a generator of 
$\cB^+(X)$ and choose an element $\tau_1 \in \cB^-(X)$. Note that $\cB^-(X) = \cB^+(X) \tau_1$ by \cite[Lemma 3.2]{LP13}.  Let $\tau_2 = \sigma \tau_1 \in \cB^-(X)$. Now, as $\tau_1^2 = \tau_2^2 = \text{id}$, we have
\begin{equation}
\sigma = \tau_2 \tau_1  \text{ and } \tau_1 = \tau_2 \sigma.
\end{equation}
Let $x \in \text{Amp}(X)$  be an integral class and consider the elements
$$ z_1 = x + \tau_1x \text{ and } z_2 = z_1 + \sigma z_1.$$
Then $z_1$ is an eigenvector of $\tau_1$ to the eigenvalue 1 and thus, up to a scalar, independent of the choice of $x$. Moreover, as $\sigma z_1 = \tau_2 z_1$ by the choice of $\tau_2$, the divisor $z_2$ is an eigenvector of $\tau_2$ to the eigenvalue 1. Thus, 
as for each $i$ we have $\tau_iz_i = z_i$ and $\tau_i$ is not automorphism,  \cite[Lemma 2.4]{LP13} implies that $z_i$ is either outside of $\Nef(X)$ or on the boundary.  For $i = 1,2$ and $k \in \mathbb{Z}$ we set 
$$\tau_{k,i} = \sigma^k\tau_i\sigma^{-k} \in \cB^{-}(X).$$
Then
$$ \tau_{k,i}(\sigma^k z_i) =
\sigma^k\tau_i\sigma^{-k} (\sigma^k z_i) = \sigma^k(\tau_iz_i) =
 \sigma^k z_i$$
 which implies, with the same argument as above,  that $\sigma^k z_i \notin \Amp(X)$ for $i = 1,2$ and all $k \in \mathbb{Z}$. 
Now, let 
$$\Pi = \mathbb{R}_{\geq 0}z_1 + \mathbb{R}_{\geq 0}z_2$$
and 
$$ \Pi \cup \tau_2 \Pi = \mathbb{R}_{\geq 0}z_1 + \mathbb{R}_{\geq 0} \sigma z_1.$$ Then, as in the case $\cB^{-}(X) = \emptyset$, we obtain
\begin{equation}\label{eq_fundamentalcone}
\text{Mov}(X) = \bigcup_{n \in \mathbb{Z}} \sigma^n(\Pi \cup \tau_2 \Pi ) = \bigcup_{g \in \cB(X)} g\Pi.
\end{equation}
The fact that $\Pi$ and $g\Pi$ have disjoint interior for $g\in \cB(X)$ follows exactly as in \cite[Theorem 4.5]{LP13}. 

Finally, we show that we can choose $\Pi$ so that $\Nef(X) \subseteq \Pi$. 
As $\Nef(X) \subseteq \Mov(X)$,  there exists a $k \in \mathbb{Z}$ such that 
$$ \sigma^k (\Pi \cup \tau_2\Pi) \cap \Nef(X) \neq \emptyset.$$
But the boundary rays $\sigma^k z_1$, $\sigma^k z_2$ and $\sigma^{k+1} z_1$ of the cones $\sigma^k\Pi$ and $\sigma^k\tau_2\Pi$ are all not contained in $\Amp(X)$. This implies either $ \Nef(X) \subseteq \sigma^k\Pi$ or $\Nef(X) \subseteq \sigma^k\tau_2\Pi$.  Now, swapping $\Pi$ and $\tau_2 \Pi$ if necessary, we obtain the desired claim by taking $\sigma^k \Pi$ as the new fundamental domain. 
\end{proof}
Visualizing the results of the proof of Proposition \ref{prop_NefInFundamental}, we obtain:
\begin{figure}[ht]
\includegraphics[scale=1]{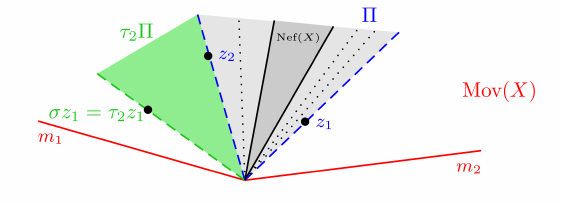} 
\caption{The movable cone $\MovC(X)$ is spanned by the rays $m_1$ and $m_2$ which are the positive eigenspaces of $\sigma$ corresponding to eigenvalues $\lambda > 0$ and $\lambda^{-1}$ (as $\det(\sigma) = 1$). In the picture, we choose $\sigma$ so that $\sigma ^n x$ for $n \rightarrow \infty$ converges (after scaling) to the ray $m_1$. The fundamental domain $\Pi$ contains the nef cone of $X$ and the green cone $\tau_2\Pi$ only exists for $\cB^-(X)\neq \emptyset$.}
\label{pictureCone}
\end{figure}

Next, we describe the fundamental domain for the action of the birational automorphism group on the movable cone in the case $\cA(X) \neq \id$. Recall that this implies $|\cA(X)| = 2$ by Theorem \ref{propertiesCY}.
	\begin{proposition}\label{prop_NefInFundamentalAut}
		Let $X$ be a smooth Calabi-Yau variety with Picard number $2$ and infinite $\Bir(X)$ such that $|\cA(X)| = 2$. 
		\begin{enumerate}
			\item[(a)] There exists a fundamental domain $\Pi'$ for the action of $\Bir(X)$ on $\MovC(X)$ and birational automorphisms $\tau_1 \in \cA^-(X)$ and $\tau_2 \in  \cB^-(X)$ each  fixing a boundary ray of $\Pi'$. 
			\item[(b)] There exists a cone $\Pi \subset \Mov(X)$ such that $\Nef(X) \subset \Pi$ and
			\begin{equation}\label{eq_coverMovAut}
			\Mov(X) = \bigcup_{n \in \mathbb{Z}} \sigma^n \Pi,
			\end{equation}
			where $\sigma$ is a generator of $\cB^+(X)$. Moreover, each boundary ray of $\Pi$ is fixed by an element of $\cB^-(X)$. 
		\end{enumerate}
	\end{proposition}
\begin{proof} We remark that the proof is similar to the proof of Proposition \ref{prop_NefInFundamental}. We only point out the necessary changes.  \\
First note that $|\cA(X)| = 2$ implies that $\cB^-(X) \neq \emptyset$ by Theorem \ref{propertiesCY}. So let $\tau_1 \in \cA(X)^-$ be the non-trivial automorphism permuting the boundary rays of $\Nef(X)$ and let $\sigma \in \cB^+(X)$ be a generator. Then $\tau_2 = \sigma \tau_1 \in \cB^-(X)$ is a non-trivial birational automorphism. Now we can run through the proof of Proposition \ref{prop_NefInFundamental} with the chosen  $\tau_1$, $\tau_2$ and $\sigma$. As before, we obtain classes $z_1$ and $z_2$ fixed by $\tau_1$ and $\tau_2$ spanning a rational polyhedral cone $\Pi'$ which is a fundamental domain for $\Bir(X)$ acting on $\MovC(X)$. However, in contrast to Proposition \ref{prop_NefInFundamental}, the ray $z_1$ is the only boundary ray which is contained in $\Amp(X)$.  	Hence, $\Nef(X)$ and $\Pi'$ have a non-empty intersection. 
	We have
$$
\bigcup_{n \in \mathbb{Z}} \sigma^n(\Pi' \cup \tau_2 \Pi' )  = \Mov(X) = \bigcup_{n \in \mathbb{Z}} \sigma^n(\Pi' \cup \tau_1 \Pi' )  
$$
where the cone $\Pi' \cup \tau_2 \Pi'$ is spanned by $z_1$ and $\sigma z_1$ and $\Pi' \cup \tau_1 \Pi'$ by $z_2$ and $\sigma^{-1}z_2$. Setting $\Pi := \Pi' \cup \tau_1 \Pi'$ we obtain Equation \eqref{eq_coverMovAut}.
Moreover,  $\Nef(X)$ is contained in the cone $\Pi$ by construction. 
\end{proof}
For Proposition \ref{prop_NefInFundamentalAut} we obtain a similar picture as Figure \ref{pictureCone} by moving the ray spanned by $z_1$ to the interior of the cone $\Nef(X)$ and substituting $\Pi$ with $\Pi'$.

\subsection*{Covering of the fundamental domain by nef cones}

Let $X$ be a smooth Calabi-Yau threefold and let $\phi:X \dasharrow X'$ be a birational map to another smooth Calabi-Yau threefold. Since $X$ and $X'$ are both minimal, the birational map $\phi$ is an isomorphism in codimension $1$, 
and hence, induces isomorphisms 
\begin{align*}
 N^1(X)  &\cong N^1(X'), \\
 \overline{\textrm{Mov}}(X) & \cong \overline{\textrm{Mov}}(X'), \\
 \Bir (X) &\cong \Bir(X'),
\end{align*}
where the last isomorphism follows from \cite[Thm. (3.7) 2.]{Han87} and is induced by conjugation with $\phi$. 

For our Calabi-Yau threefold $X$ one can cover the fundamental domain by birational pullbacks of nef cones of minimal models of $X$. Due to results of Kawamata we only have to consider \emph{small contractions} (or \emph{$D$-flops}) in the following sense.

\begin{definition}[{\cite[Definition 1.8]{Kaw97}}]
 Let $X$ be a smooth Calabi-Yau threefold, and let $D$ be an effective divisor on $X$ which is not nef. Then for a sufficiently small $\epsilon>0$ the pair $(X,\epsilon D)$ is log terminal and there exists a contraction morphism $f\colon X  \rightarrow Y$ contracting a $(K_X+\epsilon D)$-negative extremal ray $R$.
 As $K_X=0$ is nef, $f$ is a primitive birational contraction morphism. If $f$ only contracts curves we call it a \emph{small contraction}. We  call the log flip of $f$ a \emph{$D$-flop}.
  
\end{definition}

In the rest of the article, we  consider mainly birational maps $\phi: X\dasharrow X'$ which are composition of $D$-flops, and thus, isomorphisms in codimension 1.

\begin{theorem}[{\cite[Thm. 2.3]{Kaw97}}, {\cite[Thm 3.5]{Kol89}}]\label{thm_Dflops} Let $X$ be a Calabi-Yau threefold, and let $D \in \MovC(X) \cap \Eff(X)$. 
Then there exists a finite sequence of $D$-flops such that the strict transform of D becomes nef (on a different minimal model of $X$).  
\end{theorem}

As a consequence we obtain the following result:

\begin{corollary}\label{coverNef} Let $X$ be a Calabi-Yau threefold of Picard number $2$ with infinite birational automorphism group. Let $\Pi$ be the fundamental domain for the action of $\Bir(X)$ on $\MovC(X)$ from Proposition \ref{prop_NefInFundamental} in the case $\cA(X) = \id$ or the cone $\Pi$ from Proposition \ref{prop_NefInFundamentalAut} otherwise. Then there exist finitely many Calabi-Yau threefolds $X_1,\ldots,X_r$ and birational maps $\phi_i\colon X\dashrightarrow X_i$ which are isomorphisms in codimension 1 such that
 $$
 \Pi = \bigcup_{i=1}^r \phi_i^*(\Nef(X_i)).
 $$
 Each $\phi_i$ is a composition of $D$-flops.
\end{corollary}
The corollary follows by results of \cite{Kaw97} but for the convenience of the reader we give a proof using the structure of the fundamental domain $\Pi$ as constructed in Proposition \ref{prop_NefInFundamental}.

\begin{proof}
	Let $D \in \Pi$. By Theorem \ref{thm_Dflops} there exists a minimal model $X_D$ and a sequence of $D$-flops $\phi_D\colon X \dashrightarrow X_D$ such that $(\phi_D)_\ast D $ is nef on $X_D$.  Note that $X_D$ has the same analytic type of singularities as $X$ by 
	\cite[Theorem 2.4]{Kol89}.
	Thus,
	$$
	\Pi \subseteq \bigcup_{D \in \Pi} \phi_D^\ast \Nef(X_D).$$ 
	To show equality, we just have to show that none of the boundary divisors of $\Pi$ is contained in the interior of some $\phi_D^\ast \Nef(X_D)$. 
	
    For the case $\cB^{-}(X) \neq \emptyset$ we recall from the proof of Proposition \ref{prop_NefInFundamental} (resp. Proposition \ref{prop_NefInFundamentalAut}) that each boundary divisor of $\Pi$ is fixed by a birational automorphism in $ \cB^-(X)$.  We prove it here for the case $\cA(X) = \id$. The other case is shown in the same way.
	Assume that there is a birational map $\phi$ as above and minimal model $X'$ of $X$ such that $\phi_\ast z_i  = R' \in \Amp(X')$ for one $i$.  
$$
 \begin{xy}
  \xymatrix{
   X \ar@{-->}[rr]^{\tau_i} \ar@{-->}[d]_{\phi} && X \ar@{-->}[d]^{\phi} \\
   X' \ar@{-->}[rr]^{\tau_i':=\phi\circ \tau_i \circ \phi^{-1}} && X'
  }
 \end{xy}
$$	
	Now as $\tau_i z_i = z_i$ we get $(\tau_i')_\ast R' = R' \in \Amp(X')$. But this implies that $\tau_i'$ is an automorphism which is a contradiction to $\Bir (X) \cong \Bir(X')$ (where this isomorphism is given by conjugation with $\phi$).  
	
	Next let us assume that $\cB^-(X) = \emptyset$ and thus $\cA(X)=\id$. As above, let $z_1$ and $z_2$ be the generators of $\Pi$. 
	Note that by construction $z_1$  is a boundary divisor of $\Nef(X)$ and $z_2 = \sigma z_1$, where $\sigma$ generates $\cB^+(X)$.  Assume there exists an ample divisor $R'$ on a minimal model $X'$ of $X$ such that $\phi^*(R') = z_2$. But $\sigma^{-1}(\phi^*(R')) \in \Nef(X)$ which is a contradiction to \cite[Lemma 1.7]{HK00}.
	
	The fact that we can cover $\Pi$ by the nef cones of only finitely many minimal models  follows from \cite[Theorem 2.6]{Kaw97} as $\Pi \subseteq \Mov(X) = \BigCone(X)$ in our setting. 
	\end{proof}
\section{Computing the numerical dimension}
In this section we adapt the estimates for the number of global sections for divisors close to the boundary of $\Mov(X)$  from \cite{L21} to our more general setting. 

\begin{lemma}[see {\cite[Lemma 5]{L21}}] \label{lem_bignef} Let $X$ be a Calabi-Yau threefold with Picard number 2. 
 There exist constants $C_{1,1}, C_{2,1} > 0$ such that for any big and nef divisor  Cartier divisor $D$ on $X$ we have
	$$ C_{1,1}D^3 < h^0(X,D) < C_{2,1}D^3.$$
\end{lemma}
\begin{proof}
	Let $D = a_1D_1+a_2D_2$ be a big and nef divisor, where the classes of $D_1$ and $D_2$ generate $\Nef(X)$. 
	Using the Hirzebruch–Riemann–Roch theorem and the Kawamata-Viehweg vanishing theorem we obtain
	$$ h^0(X,D) = \chi(X,D) = \frac{D^3}{6} + \frac{c_2(X)\cdot D}{12},$$
	and thus, as $D^3 \neq 0$,
	$$ \frac{h^0(X,D)}{D^3} =  \frac{1}{6} + \frac{c_2(X)\cdot D}{12D^3},$$
Now $D^3$ is a cubic polynomial in $a_1$ and  $a_2$ which takes only positive values. By results of Miyaoka  (\cite{miyaoka1987chern}, Theorem 6.6), $c_2(X)\cdot D$ is non-negative for any nef divisor $D$ and  a linear polynomial in $a_1$ and $a_2$. Thus, the right-hand-side is bounded from above by a positive constant $C_{2,1}$, whereas as a lower bound we may choose $C_{1,1} = \frac{1}{7}$.     
	\end{proof} 
	                     
Next we recall some notation from \cite{L21} which is used in the following. Let $X$ be as in Assumption \ref{mainassumption}. For a  divisor
$$D  = a_1R_1 + a_2R_2 \in \text{Mov}(X) (= \text{Big}(X)),$$
where $R_1\in m_1$ (resp. $R_2\in m_2$) is a fixed eigenvector to the eigenvalue $\lambda > 1$ (resp. $\lambda^{-1}$) of $\sigma$  we introduce new quadratic forms (in $a_1$ and $a_2$)
$$ L_1(D) = a_1a_2 \text{ and } L_2(D)= \frac{a_1}{a_2}.$$
With respect to the basis $R_1, R_2$, the matrix $\sigma \in \cB^{+}(X)$ is of the form 
$$
\begin{pmatrix}
\lambda & 0 \\ 0 & \lambda^{-1}
\end{pmatrix},    
$$
and we have
\begin{equation}\label{eq_invariantsigma}
L_1(\sigma D) = L_1(D) \text{ and }        L_2(\sigma D)   = \lambda^2L_2(D).
\end{equation}

\begin{remark} \label{L12coordinates}
	\begin{itemize}	\item[(a)] $L_1(-)$ is a quadratic form invariant under $\sigma$ and $L_2(-)$ describes the slope of rays. Therefore, a subcone of $\text{Big}(X)$ corresponds to a horizontal strip in the upper right corner of the new coordinate system given by $L_1(-)$ and $L_2(-)$. 
\item[(b)] For $D = a_1R_1 + a_2R_2$ big, we get $a_1 = L_1(D)^\frac{1}{2}L_2(D)^\frac{1}{2}$ and 
$a_2 =L_1(D)^\frac{1}{2}L_2(D)^{-\frac{1}{2}}$.
	\end{itemize}	
\end{remark}

\begin{lemma}\label{ineqL1} 
Let $X$ be a  smooth Calabi-Yau threefold of Picard number $2$ with infinite birational automorphism group such that $\cB^-(X)\neq \emptyset$.  
 Let $\Pi$ be the fundamental domain for the action of $\Bir (X)$ on $\MovC (X)$ as constructed in Proposition \ref{prop_NefInFundamental} and let $\tau_2\in \cB^-(X)$ be the birational automorphism fixing the boundary ray $z_2$ of $\Pi$. 
    Then 
	$$ \lambda^{-1} L_1(D) < L_1(\tau_2D )<\lambda L_1(D).$$
	for all divisors $D\in \Pi$ in the fundamental domain. 
\end{lemma}	
\begin{proof}
 Using the notation of Proposition \ref{prop_NefInFundamental}, the fundamental domain is generated by the two rays $z_1 = a_1R_1 + a_2R_2$ and $z_2 = b_1R_1 + b_2R_2$. Let $D=d_1z_1+d_2z_2$ be a divisor in the fundamental domain with positive coefficients $d_1,d_2$. Since $\tau_2 z_2 = z_2$ and $\tau_2 z_1 = \sigma z_1$, we have 
 $$
 \tau_2 D = d_1 \sigma z_1 + d_2 z_2 = 
 (\lambda d_1a_1 + d_2b_1)R_1 + (\lambda^{-1}d_1a_2 + d_2b_2)R_2.
 $$
 We compute 
 $$
 L_1(D) = a_1a_2 d_1^2 + (a_1b_2 + a_2b_1)d_1d_2 + b_1b_2d_2^2
 $$
 and 
 $$
 L_1(\tau_2D) = a_1a_2 d_1^2 + (\lambda a_1b_2 + \lambda^{-1}a_2b_1)d_1d_2 + b_1b_2d_2^2.
 $$
 Note that $L_1(D)$ and $L_1(\tau_2D)$ are homogeneous polynomials in $d_1$ and $d_2$ of degree $2$. Since $\lambda>1$ and all coefficients of the polynomials $L_1(D)$ and $L_1(\tau_2D)$ are positive, we get 
 $$
 \lambda^{-1} L_1(D)< L_1(\tau_2D) < \lambda L_1(D)
 $$
 for all positive integer pairs $(d_1,d_2)$. 
\end{proof}
\begin{lemma}\label{lem_bignef_contraction}
Let $X$ be a  smooth Calabi-Yau threefold of Picard number $2$ with infinite birational automorphism group.  Let $\Pi$ be the fundamental domain for the action of $\Bir (X)$ on $\MovC (X)$ as constructed in Proposition \ref{prop_NefInFundamental} in the case $\cA(X) = \id$ or the cone $\Pi$ from Proposition \ref{prop_NefInFundamentalAut} otherwise.
There exist constants $C_{1,2},C_{2,2} >0$ such that for any big Cartier divisor  $D \in \Pi$
$$ C_{1,2}L_1(D)^{\frac{3}{2}} < h^0(X,D) < C_{2,2}L_1(D)^{\frac{3}{2}}.$$
\end{lemma}
\begin{proof}
Let $D\in \Pi$ be a big Cartier divisor. 
By Corollary \ref{coverNef} there exists a minimal model $X'$ of $X$ and an isomorphism in codimension one $\phi\colon X \dashrightarrow X'$ such that $\phi_\ast D = D'$ is nef.  Note that $X'$ is a smooth Calabi-Yau variety satisfying the properties of Assumption \ref{mainassumption}. Now, as $h^0(X,D) = h^0(X',D')$, Lemma \ref{lem_bignef} implies
	$$ C_{1,1}D'^3 < h^0(X,D) < C_{2,1}D'^3.$$
Hence, it remains to bound $D'^3$ from above and below with respect to $L_1(D)$. Recall that $\phi$ induces an isomorphism $ \MovC(X) \cong \MovC(X')$, hence, the strict transforms  $R_1' = \phi_\ast R_1$ and $R_2' = \phi_\ast R_2$ generate $\MovC(X')$.  Consequently, for the divisor 
 $D = a_1R_1+a_2R_2 \in \textrm{Big} (X) = \Mov(X)$  with positive coefficients $a_1,a_2$ 
 we obtain 
 $$ D' = a_1R_1'+a_2R_2'$$
and 
using the coordinates $L_1(-)$ and $L_2(-)$ also on $\MovC(X') \cong \MovC(X)$, we get 
$$ L_1(D') = L_1(D) \text{ and }  L_2(D') = L_2(D).$$ 
Now, as $D'^3$ is positive for all $D' \in \Nef(X')$ as well as $L_1(D')$ and $L_2(D')$ by definition, we get
\begin{align*}
0 < D'^3  = &\  a_1^3 R_1'^3 + a_1^2a_2R_1'^2R_2' + a_1a_2^2R_1'R_2'^2 + a_2^3 R_2'^3 \\
   = &\ L_1(D')^\frac{3}{2}L_2(D')^\frac{3}{2} R_1'^3 +
    L_1(D')^\frac{3}{2}L_2(D')^\frac{1}{2} R_1'^2R_2'  \\
   & + L_1(D')^\frac{3}{2}L_2(D')^{-\frac{1}{2}} R_1'R_2'^2 + 
   L_1(D')^\frac{3}{2}L_2(D')^{-\frac{3}{2}} R_2'^3 \\
   = & \ L_1(D')^\frac{3}{2}\underbrace{\left(L_2(D')^\frac{3}{2} R_1'^3 +
    L_2(D')^\frac{1}{2} R_1'^2R_2'  
    + L_2(D')^{-\frac{1}{2}} R_1'R_2'^2 + 
   L_2(D')^{-\frac{3}{2}} R_2'^3\right)}_{>0, \textrm{ for all }D'\in \Nef (X')}.
\end{align*}  
But $L_2(-)$ is bounded from below and above for any divisor in $\phi^* \Nef(X')$ by Remark \ref{L12coordinates} and the intersection products of $R_1'$ and $R_2'$ only depend on $\phi$. Therefore, we can bound the factor 
$$L_2(D')^\frac{3}{2} R_1'^3 +
    L_2(D')^\frac{1}{2} R_1'^2R_2'  
    + L_2(D')^{-\frac{1}{2}} R_1'R_2'^2 + 
   L_2(D')^{-\frac{3}{2}} R_2'^3$$
from below and above for any divisor in $D'\in \Nef (X')$ only depending on $\phi$.
Hence, we obtain positive constants $C_{1,\phi}$ and $C_{2,\phi}$ for all $D\in \phi^* \Nef(X')$ such that 
$$ C_{1,\phi}L_1(D)^{\frac{3}{2}} < h^0(X,D) < C_{2,\phi}L_1(D)^{\frac{3}{2}}.$$ 
The result follows now from the fact that we can cover $\Pi$ by only finitely many $\phi_i^*\Nef(X_i)$ as shown in Corollary \ref{coverNef}. 
\end{proof}

\begin{corollary}\label{ineqL1h0} 
Let $X$ be a  smooth Calabi-Yau threefold of Picard number $2$ with infinite birational automorphism group. 
 There exist constants $C_{1,3},C_{2,3} >0$ such that for any Cartier divisor $D \in \Mov(X)$
	$$ C_{1,3}L_1(D)^{\frac{3}{2}} < h^0(X,D) < C_{2,3}L_1(D)^{\frac{3}{2}}.$$
\end{corollary}	
\begin{proof}
   Since  the quadratic form $L_1(-)$ and $h^0(X,-)$ are invariant under the birational automorphism $\sigma$ which is an isomorphism in codimension $1$,  we may assume from the proof of Proposition \ref{prop_NefInFundamental}
    that $D\in \Pi\cup \tau_2\Pi$ in the case $\cA(X) = \id$ or $D \in \Pi$ for the cone $\Pi$ from Equation \eqref{eq_coverMovAut} in the case  $\cA(X) \neq \id$. If $D \in \Pi$ the result follows directly from the previous lemma. 	For $D \in \tau_2\Pi$, we can apply Lemma \ref{ineqL1} and proceed as before. 
	\end{proof}
	
\begin{remark}
	Note that in the proof of Corollary \ref{ineqL1h0} we work with the fundamental domain $\Pi$ which is the finite union of pullbacks of nef cones as shown in Corollary  \ref{coverNef}. But the computation of the numerical dimension of a boundary divisor of $\MovC(X)$ also works for any closed cone spanned by a divisor $R$ and $\sigma R$ where $\sigma$ generates $\cB^+(X)$.
\end{remark}

  \begin{lemma}[see {\cite[Lemma 7]{L21}}] \label{rounddown}
  Let $X$ be a  smooth Calabi-Yau threefold of Picard number $2$ with infinite birational automorphism group. 
  	There exist constants $C_{1,4},C_{2,4}$ and $C_1 > 1$ such that for any  $D = a_1R_1 + a_2R_2 \in \MovC(X)$ and ample divisor $A =  b_1R_1 + b_2R_2$ with $b_i > C_1$ we have
  	$$ C_{1,4}L_1(D+A) < L_1(\lfloor D\rfloor +A) < C_{2,4}L_1(D+A) .$$
  \end{lemma}
\begin{proof}
	The proof is identical to \cite[Lemma 7]{L21}.
	\end{proof}

\begin{theorem}[$\Rightarrow$ Main Theorem]
Let $X$ be a  smooth Calabi-Yau threefold of Picard number $2$ with infinite birational automorphism group. 
 Suppose that $A = b_1 R_1  + b_2R_2$ is an ample Cartier divisor with $b_1,b_2 \geq C_1$. There exist constants $C_{1,5}$ and $C_{2,5}$ such that for all sufficiently large $m$
 $$
 C_{1,5}m^{\frac{3}{2}} < h^0 (X, \lfloor mR_1\rfloor + A) < C_{2,5} m^{\frac{3}{2}}.
 $$
\end{theorem}

\begin{proof}
 By Lemma \ref{rounddown}, we find constants $C_{1,4},C_{2,4}$ and $C_1 > 1$ such that 
 $$ C_{1,4}L_1(mR_1 +A) < L_1( \lfloor mR_1\rfloor + A) < C_{2,4}L_1(mR_1 +A)$$ 
 for the given ample divisor $A$. We can compute 
 $$
 L_1(mR_1 +A) = L_1((m+b_1)R_1 + b_2R_2) = (m+b_1)b_2.
 $$
 Hence, combining this with Lemma \ref{ineqL1h0} for $D = \lfloor mR_1 \rfloor +A$, we get 
 $$
 C_{1,3}(C_{1,4}(m+b_1)b_2)^{\frac{3}{2}} < h^0 (X, \lfloor mR_1\rfloor + A) < C_{2,3}(C_{2,4}(m+b_1)b_2)^{\frac{3}{2}}
 $$
 and for sufficiently large $m$ the theorem follows.
\end{proof}

\section{Examples}\label{examples}

We briefly review known Calabi-Yau threefolds with Picard number $2$ and infinite birational automorphism group. 
Examples of Calabi-Yau threefolds as in Assumption \ref{mainassumption} can be constructed similar as in \cite{Og14} as complete intersections on $\PP^3\times \PP^3$ using \cite{Y21}.
In \cite[Section 6.2]{LW21}, there is a further example of a Calabi-Yau threefold as in Assumption \ref{mainassumption} such that there exist birational involutions. A similar example of a Calabi-Yau threefold with Picard number two and infinite birational automorphism group but without any birational involution is given in \cite{HT18} (see also \cite[Section 6.1]{LW21}).

At the end we construct new examples of a Calabi-Yau threefold $X$ with Picard number two and infinite birational automorphism group such that there exist birational involutions.  Moreover, we adapt the original example of Oguiso \cite{Og14} to one with an automorphism group that acts non-trivially on the nef cone. This shows that the case of Proposition \ref{prop_NefInFundamentalAut} can indeed occur.
We make all computations explicit using the computer algebra program \emph{Macaulay2} \cite{macaulay2}.  All of our computations are available on the authors' homepage \cite{CY3foldM2}.

\begin{example}
	Recall that the example of \cite{Og14} is a Calabi-Yau threefold $X$which is a complete intersection of forms of bidegree $(1,1)$, $(1,1)$ and $(2,2)$ in $\PP^3 \times \PP^3$. The two resulting projections $\pi_i\colon X \rightarrow \PP^3$ are generically 2:1 and induce involutions $\tau_i \colon X \dashrightarrow X$. In the basis $h_1$ and $h_2$ with $h_i = \pi_i^* (H_{\PP^3})$, we have
	 $$ \tau_1^* = \begin{pmatrix}
	 1 & 6 \\ 0 & -1
	 \end{pmatrix}
	 \text{ and }
	 \tau_2^* = \begin{pmatrix}
	 -1 & 0 \\ 6 & 1
	 \end{pmatrix}.
	 $$ 
	 Let $x_i$ be the variables of the first $\PP^3$, and $y_j$ the variables of the second one. 
	 Now we choose the defining forms of $X$ so that the generated ideal is invariant under interchanging the variables $x_i$ and $y_i$. More precisely, we consider the homomorphism $g$ of the Cox ring of $\PP^3 \times \PP^3$ given by
	 $$ x_i \mapsto y_i, \  y_i \mapsto x_i.$$
	 Then we choose a generic $(1,1)$-form $f_0$ and a generic $(2,2)$-form $q$ and consider the ideal generated by $$f_0, g(f_0), q + g(q)$$ which is invariant under $g$. Thus, we obtain an automorphism $\alpha$ on the corresponding variety $X \subset \PP^3 \times \PP^3$ which interchanges  the variables $x_i$ and $y_i$. We refer to our ancillary \emph{Macaulay2}-file \cite{CY3foldM2} where we checked in an example that such a variety $X$ is a smooth Calabi-Yau threefold.
	 In the basis $h_1$ and $h_2$, we have 
	 $$ \alpha^* = \begin{pmatrix}
	 0& 1 \\ 1 & 0
	 \end{pmatrix}.
	 $$
	 Moreover, note that 
	 $$\tau_2 ^*= \alpha^* \tau_1^* \alpha^*.$$ 
	\end{example}	
	 
Next let us consider the product $\PP^3\times \PP^5$ with its two projections $\pi_1$ and $\pi_2$
$$
\xymatrix{
	&\PP^3\times \PP^5 \ar[ld]_{\pi_1} \ar[rd]^{\pi_2} & \\
	\PP^3 & & \PP^5.
}
$$  
We denote by $R = \mathbb{Q}[x_0,\dots,x_3,y_0,\dots,y_5]$ the bigraded Cox ring of $\PP^3\times \PP^5$ where $x_0,\dots,x_3$ and $y_0,\dots,y_5$ are the coordinate functions on $\PP^3$ and $\PP^5$, respectively. Let $H_1 = \pi_1^*(H_{\PP^3})$ and $H_2= \pi_2^*(H_{\PP^5})$ be the pullbacks of the hyperplane classes.  Let us consider the 5-fold $Z$ given by the $4\times 4$ Pfaffians of the following skew-symmetric matrix 
$$
M = \begin{pmatrix}
0 & x_0 &x_1 & x_2 & x_3 \\
-x_0 & 0 & y_0 & y_1 & y_2 \\
-x_1 & -y_0 & 0 & y_3 & y_4 \\
-x_2 & -y_1 & -y_3 & 0 & y_5 \\
-x_3 & -y_2 & -y_4 & -y_5 & 0 
\end{pmatrix},
$$
that is, the ideal of $Z$ is generated by
\begin{align*}
I_Z = \langle & {y}_{0}y_5 -{y}_{1}{y}_{4}+{y}_{2}{y}_{3
}, \\
& {x}_{1}y_5-{x}_{2}{y}_{4}+{x}_{3}{y}_{3},\ {x}_{0}y_5-{x}_{2}{y
}_{2}+{x}_{3}{y
}_{1}, \\
& {x}_{0}{y}_{4}-{x}_{1}{y
}_{2}+{x}_{3}{y}_{0},\ {x}_{0}{y
}_{3}-{x}_{1}{y}_{1}+{x}_{2}{y
}_{0}\rangle 
\end{align*}
whose syzygy matrix is $M$. 
The image of $Z$ in $\mathbb{P}^5$ by the second projection $\pi_2$ is  the Grassmannian $G(2,4)\subset \PP^5$ given by the Pl\"ucker equation 
$$y_0y_5  - y_1y_4 + y_2y_3.$$

One can check that $Z$ is smooth, and by the structure theorem of Buchsbaum--Eisenbud \cite[Theorem 2.1]{BE77},  the maximal Pfaffians of $M$ define a Gorenstein variety of codimension $3$. From the structure of the minimal bigraded free resolution of $Z$ (see \cite{CY3foldM2}),  the canonical bundle is $K_Z=(-2H_1 - 3H_2)|_Z$. 

Next we show that the first projection of $\PP^3\times \PP^5$ induces a $\PP^2$-bundle structure on $Z$. Indeed, the fiber of $\pi_1$ over a point $P\in \PP^3$ is given by four linear forms in $\PP^5$ and the quadric generator of $G(2,4)\subset \PP^5$. But, the four linear forms have exactly one constant syzygy, that is, the first row of $M$ evaluated at $P$.  Hence, the fiber over $P$ is at most a plane generated by three linear forms in $\PP^5$.  By semi-continuity of the fiber dimension it follows that the fiber is exactly a plane.  

Similarly one can show that the second projection $\pi_2$ of $Z$ to the Grassmannian $G(2,4)$ induces a $\PP^1$-bundle structure on $Z$.  
We note that the matrix $M$ evaluated at a point of the Grassmannian $G(2,4)\subset \PP^5$ generates four linear forms on $\PP^3$. But these linear forms have exactly two constant syzygies (since the lower-right $4\times 4$ matrix of $M$ has rank $2$ at every point of the Grassmannian), whence, generate a line in $\PP^3$.

\begin{remark}
 The variety $Z$ can be regarded as a bigraded version of a generic syzygy variety as described in \cite{ES12}.  Furthermore, from the bundle structures explained above, $Z$ can be seen either as  the projectivisation of the cotangent sheaf $\Omega_{\PP^3}$ on $\PP^3$ or as the projectivisation of the universal rank $2$ quotient bundle on the Grassmannian $G(2,4)$.
\end{remark}

Now, as $K_Z = (-2H_1 - 3H_2)|_Z$, intersecting $Z$ with two hypersurfaces whose bidegrees add up to $(2,3)$, we obtain in general a Calabi-Yau threefold. There are several possibilities for such pairs and in this article we study exemplary the case of a complete intersection of $Z$ with forms of bidegree $(0,1)$ and $(2,2)$ and one with forms of bidegree $(1,2)$ and $(1,1)$. 

\begin{example}[The case $(0,1)$ and $(2,2)$]\label{example1}
Intersecting $Z$ with a general hypersurface of bidegree $(0,1)$ gives  a ruled Fano $4$-fold $Y$ in $\PP^3 \times \PP^4$ of index $2$ with two $\PP^1$-bundle structures.  By \cite{Wis89} (see also \cite[Example 2 ($g=13$)]{mukai-biregularclassification}),  the ruled Fano $4$-fold $Y$ is unique and the two $\PP^1$-bundle structures are induced by the null-correlation bundle on $\PP^3$ (that is,  a stable rank $2$ bundle with $c_1=0$ and $c_2=1$) or a stable rank $2$ bundle on a three dimensional quadric $Q_3\subset \PP^4$ with $c_1=-1$ and $c_2=1$.  Our Calabi-Yau threefold $X$ is a smooth anticanonical divisor of $Y$ where two birational involutions are induced by the two $\PP^1$-bundle structures on $Y$.  The variety $Y\subset \PP^3\times \PP^4$ is a Fano $4$-fold of index $2$ since its canonical bundle is given as the restriction of $-2(H_1 +H_2)$ to $Y$ (where $(H_1+H_2)|_Y$ is an ample line bundle on $Y$). 
 We remark that the Segre embedding of $\PP^3\times \PP^4$ embeds $Y$ into $\PP^{15} = \PP (H^0(Y, (H_1+H_2)|_Y)^*)$, whence it is a Fano $4$-fold of genus $13$. 

By \cite[Theorem (0.1)]{Wis89}, the Fano $4$-fold $Y$ is unique as described above and has Picard number $2$ by \cite{SW90}.  Now,  a smooth intersection of $Y$ and a quadric defined by a form of bidegree $(2,2)$ on $\PP^3\times \PP^4$ is a Calabi-Yau threefold $X$ of Picard number $2$ with finite automorphism group by \cite[Theorem 3.1]{CO15}. In the following, we fix the two forms of bidegree $(0,1)$ and $(2,2)$:
\begin{align*}
&  y_0 + y_1 + y_2 +y_3 + y_4 - y_5, \\
& {x}_{0}{x}_{3}{y}_{0}{y}_{1}+{x}_{1}{x}_{2}{y}_{1}{y}_{2}-{x}_{2}^{2}{y}_{1}{y}_{2}+{x}_{0}{x}_{2}{y}_{2}{y}_{3}+{x
}_{2}^{2}{y}_{2}{y}_{3} \\
& +{x}_{1}{x}_{2}{y}_{0}{y}_{4}+{x}_{0}{x}_{1}{y}_{1}{y}_{4}+{x}_{0}{x}_{3}{y}_{1}{y}_{4}-{x}_{1
}{x}_{3}{y}_{1}{y}_{4}+{x}_{2}{x}_{3}{y}_{1}{y}_{4}  \\ & +2{x}_{1}{x}_{2}{y}_{2}{y}_{4}-{x}_{1}^{2}{y}_{3}{y}_{4}+2{x
}_{1}{x}_{2}{y}_{3}{y}_{4}+{x}_{0}^{2}{y}_{4}^{2}-{x}_{1}^{2}{y}_{4}^{2}+{x}_{2}^{2}{y}_{4}^{2}.
\end{align*}

We see that $X$ is generically $2:1$ to $\PP^3$ and $Q_3$, the corresponding quadric in $\PP^4$. In our ancillary \emph{Macaulay2}-file, we also check that there exists (at least) a line $L_i$ on $X$ for the projection $\pi_i$ ($i=1,2$) such that $\pi_i(L_i)$ is contracted to a point (for example the fibers of $Y$ over $(1:1:0:0) \in \mathbb{P}^3$ and $(0:0:0:0:1) \in Q_3$ are contained in $X$).  We denote by $\tau_i$ the birational involution induced by the projection $\pi_i$ for $i=1,2$.

We also denote by $h_1 = \pi_1^*(H_{\PP^3})|_X$ and $h_2= \pi_2^*(H_{\PP^4})|_X$ the pullbacks of the hyperplane classes restricted to $X$.  
Again, using \emph{Macaulay2}, we compute that 
$$ h_1^3 = 2, \ h_1^2h_2 = 6, \ h_1h_2^2 = 8 \text{ and } h_2^3 =2.$$
Next we compute the pullbacks of $h_i$ under the birational automorphisms $\tau_1$ and $\tau_2$. First note,  that we have $\tau_i^*h_i = h_i$ by construction. Similarly, as in \cite{Og14} we compute 
$$ t_1^*h_2 = 6h_1 - h_2 \text{ and } t_2^*h_1 = -h_1+8h_2.$$
Thus, with respect to the basis $h_1,h_2$, we have
$$ \tau_1^* = \begin{pmatrix}
1 & 6 \\ 0 & -1
\end{pmatrix}
\text{ and }
\tau_2^* = \begin{pmatrix}
-1 & 0 \\ 8 & 1
\end{pmatrix}.
$$
We then obtain a birational automorphism $\sigma = \tau_1\tau_2$ with matrix
$$ \sigma^* = \tau_2^*\tau_1^* = \begin{pmatrix}
-1 & -6 \\ 8 & 47
\end{pmatrix}.
$$
We see that $\sigma^*$ is of infinite order with eigenvalues $\lambda = 23 + 4\sqrt{33} > 1 $ (resp. $\lambda^{-1} = 23 -4\sqrt{33}$) corresponding to eigenvectors 
$ \frac{1}{2}(-6+\sqrt{33})h_1 + h_2$ (resp. $h_1 +  \frac{1}{2}(-6+\sqrt{33})h_2$). 

By Proposition \ref{prop_NefInFundamental} we see that the nef cone of $X$ spanned by $h_1$ and $h_2$ is a fundamental domain of the action of the birational automorphism group on the movable cone of $X$ spanned by the eigenvectors of $\sigma$. 

We conclude that we have found a Calabi-Yau threefold with Picard number $2$ and infinite birational automorphism group.  We could not find a description of this Calabi-Yau threefold in the literature.  
\end{example}

\begin{example}[The case $(1,2)$ and $(1,1)$]
Intersecting $Z$ with a general hypersurface of bidegree $(1,2)$ gives  a Fano $4$-fold $Y$ in $\PP^3 \times \PP^5$ of index $1$ where the first projection $\pi$  induces a conic-bundle structure on $Y$.  We compute that the restriction of $\pi_2$ to $Y$ is a divisorial contraction which contracts a divisor in $Y$ to a smooth surface in $G{(2,4)} \subseteq \PP^5$ of degree 13 and sectional genus 14. 

Now,  a smooth intersection of $Y$ with a form of bidegree $(1,1)$ on $\PP^3\times \PP^5$ is a Calabi-Yau threefold $X$ of Picard number $2$ with finite automorphism group by \cite[Theorem 3.1]{CO15}. 
In  our ancillary \emph{Macaulay2}-file, we check that the first projection $\pi_1$ to $\PP^3$ is $2:1$ and induces a birational involution as in the first example. The second projection $\pi_2$ maps $X$ birationally to a singular complete intersection $V_{2,4} \subseteq \PP^5$ of the Grassmannian $G(2,4)$ and a quartic hypersurface. 
The complete intersection $V_{2,4}$ is singular in $33$ double points, denoted by $\Sigma$, which lie on a unique Del Pezzo quintic surface $S$.  We may apply \cite[Lemma 2.14]{LW21} to $V_{2,4}$:
\begin{itemize}
 \item The projection $\pi_2$ of $X$ to $V_{2,4}$ is a small resolution and the restriction to $\pi_2^{-1}(S) \rightarrow S$ is the blow-up of $S$ in $\Sigma$. 
Hence, we get $33$ lines $\pi_2^{-1}(\Sigma)$ on $X$ which are contracted by $\pi_2$. Note that by our construction $X$ is isomorphic to the blow-up $\textrm{Bl}_S(V_{2,4})$ of $V_{2,4}$ along the surface $S$. 
\item There exists an Atiyah flop $X^+ \rightarrow V_{2,4}$ flopping the contracted curves $\pi_2^{-1}(\Sigma)$. This is obtained by the blow-up of $V_{2,4}$ in its singular points, denoted by $W= \textrm{Bl}_\Sigma(V_{2,4})$, and the following diagram 
$$
\xymatrix{
&\ \ \  \color{blue}{\ell_p\times \ell_p^+\cong \PP^1\times \PP^1} \ar@[blue]@{|->}[dl] & W = \textrm{Bl}_\Sigma(V_{2,4}) \ar[dr] \ar[dl]& \color{blue}{\ell_p\times \ell_p^+} \ar@[blue]@{|->}[dr]& \\
\color{blue}{\ell_p} \ar@[blue]@{|->}[dr] & X\cong \textrm{Bl}_S(V_{2,4}) \ar[dr]_{\pi_2} & & X^+ \ar[dl] \ar[dr]^{\varphi} & \color{blue}{\ell_p^+} \\
&\color{blue}{p\in \Sigma} & V_{2,4} & & X'
}
$$
where the morphism $W\to X^+$ contracts the strict transform of the $33$ lines $\pi_2^{-1}(\Sigma)$ on $W$. 
\item Furthermore, there is a smooth surface $S^+$ on $X^+$ isomorphic to $S$. 
\end{itemize}
Since $X$ and $X^{+}$ are isomorphic in codimension one, there exists another birational primitive contraction $\varphi$ as in the diagram above. Assuming that $\varphi$ is divisorial, we can conclude that $\Bir(X)$ is finite.
\end{example}

\begin{remark}
	The above constructions can be applied to other complete intersections of $Z$ 
producing more examples of Calabi-Yau varieties with different geometric behavior (see our \emph{Macaulay2}-file \cite{CY3foldM2}). For example, the intersection of $Z$ with two forms of bidegree $(2,1)$ and $(0,2)$ is similar to Example \ref{example1}. All of our examples are not in the list of Calabi-Yau threefolds in \cite{KK09} and \cite{Kap09}. 
\end{remark}

\begin{remark}
 All our computations will also be implemented in the new computer algebra program \emph{OSCAR} \cite{oscar}. 
\end{remark}


\begin{thebibliography}{{OSC}21}

\bibitem[Bir67]{Bir67}
Garrett {Birkhoff,} \emph{{Linear transformation with invariant cones}}, {Am. Math. Mon.}, \textbf{74} (1967), 274-276.  

\bibitem[BE77]{BE77}
David~A. {Buchsbaum} and David {Eisenbud}, \emph{{Algebra structures for finite
  free resolutions, and some structure theorems for ideals of codimension 3}},
  {Am. J. Math.} \textbf{99} (1977), 447--485 (English).

\bibitem[CO15]{CO15}
Serge {Cantat} and Keiji {Oguiso}, \emph{{Birational automorphism groups and
  the movable cone theorem for Calabi-Yau manifolds of Wehler type via
  universal Coxeter groups}}, {Am. J. Math.} \textbf{137} (2015), no.~4,
  1013--1044 (English).

\bibitem[CP21]{CP21}
Sung~Rak Choi and Jinhyung Park, \emph{Comparing numerical {I}itaka dimensions
  again}, preprint: \url{https://arxiv.org/pdf/2111.00934.pdf}, 2021.

\bibitem[{Eck}16]{Eckl16}
Thomas {Eckl}, \emph{{Numerical analogues of the Kodaira dimension and the
  abundance conjecture}}, {Manuscr. Math.} \textbf{150} (2016), no.~3-4,
  337--356 (English).

\bibitem[ES12]{ES12}
Friedrich {Eusen} and Frank-Olaf {Schreyer}, \emph{{A remark on a conjecture of
  Paranjape and Ramanan}}, Geometry and arithmetic. Based on the conference,
  Island of Schiermonnikoog, Netherlands, September 2010, Z\"urich: European
  Mathematical Society (EMS), 2012, pp.~113--123 (English).

\bibitem[{Fuj}20]{Fuj20}
Osamu {Fujino}, \emph{{Corrigendum to: ``On subadditivity of the logarithmic
  Kodaira dimension''}}, {J. Math. Soc. Japan} \textbf{72} (2020), no.~4,
  1181--1187 (English).

\bibitem[GS21]{macaulay2}
D.~R. Grayson and M.~E. Stillman, \emph{{\sc Macaulay2} --- {A} software system
  for research in algebraic geometry (version 1.18)}, home page:
  \url{http://www.math.uiuc.edu/Macaulay2/}, 2021.

\bibitem[{Han}87]{Han87}
Masaki {Hanamura}, \emph{{On the birational automorphism groups of algebraic
  varieties}}, {Compos. Math.} \textbf{63} (1987), 123--142 (English).

\bibitem[HK00]{HK00}
Yi~{Hu} and Sean {Keel}, \emph{{Mori dream spaces and GIT.}}, {Mich. Math. J.}
  \textbf{48} (2000), 331--348 (English).

\bibitem[HS21]{CY3foldM2}
Michael Hoff and Isabel Stenger, \emph{{Ancillary M2-file to "On the numerical
  dimension of {C}alabi-{Y}au 3-folds of {P}icard number 2"}}, source code
  available at \url{https://www.math.uni-sb.de/ag/lazic/stenger/index.html},
  2021.

\bibitem[HT18]{HT18}
Shinobu {Hosono} and Hiromichi {Takagi}, \emph{{Movable vs monodromy nilpotent
  cones of Calabi-Yau manifolds}}, {SIGMA, Symmetry Integrability Geom. Methods
  Appl.} \textbf{14} (2018), paper 039, 37 (English).

\bibitem[{Kap}09]{Kap09}
Grzegorz {Kapustka}, \emph{{Primitive contractions of Calabi-Yau threefolds.
  II}}, {J. Lond. Math. Soc., II. Ser.} \textbf{79} (2009), no.~1, 259--271
  (English).

\bibitem[{Kaw}88]{Kaw88}
Yujiro {Kawamata}, \emph{{Crepant blowing-up of 3-dimensional canonical
  singularities and its application to degenerations of surfaces}}, {Ann. Math.
  (2)} \textbf{127} (1988), no.~1, 93--163 (English).

\bibitem[{Kaw}97]{Kaw97}
\bysame, \emph{{On the cone of divisors of Calabi-Yau fiber spaces}}, {Int. J.
  Math.} \textbf{8} (1997), no.~5, 665--687 (English).

\bibitem[KK09]{KK09}
Grzegorz {Kapustka} and Micha{\l} {Kapustka}, \emph{{Primitive contractions of
  Calabi-Yau threefolds. I}}, {Commun. Algebra} \textbf{37} (2009), no.~2,
  482--502 (English).

\bibitem[{Kol}89]{Kol89}
J\'anos {Koll\'ar}, \emph{{Flops}}, {Nagoya Math. J.} \textbf{113} (1989),
  15--36 (English).

\bibitem[{Leh}13]{Leh13}
Brian {Lehmann}, \emph{{Comparing numerical dimensions}}, {Algebra Number
  Theory} \textbf{7} (2013), no.~5, 1065--1100 (English).

\bibitem[Les21]{L21}
John Lesieutre, \emph{Notions of numerical {I}itaka dimension do not coincide},
  J. Alg. Geom., electronically published on Feb. 02, 2021,
  \url{https://doi.org/10.1090/jag/763}, 2021.

\bibitem[LP13]{LP13}
Vladimir {Lazi\'c} and Thomas {Peternell}, \emph{{On the cone conjecture for
  Calabi-Yau manifolds with Picard number two}}, {Math. Res. Lett.} \textbf{20}
  (2013), no.~6, 1103--1113 (English).

\bibitem[LW22]{LW21}
Ching-Jui {Lai} and Sz-Sheng {Wang}, \emph{{The movable cone of certain
  Calabi-Yau threefolds of Picard number two}}, {J. Pure Appl. Algebra}
  \textbf{226} (2022), no.~2, 40 (English), Id/No 106841.

\bibitem[Miy87]{miyaoka1987chern}
Yōichi Miyaoka, \emph{The {C}hern classes and {K}odaira dimension of a minimal
  variety}, Algebraic geometry, Sendai, 1985, Mathematical Society of Japan,
  1987, pp.~449--476.

\bibitem[Muk89]{mukai-biregularclassification}
S.~Mukai, \emph{Biregular classification of {F}ano 3-folds and {F}ano manifolds
  of coindex 3}, Proc. Natl. Acad. Sci. USA \textbf{86} (1989), no.~9,
  3000--3002.

\bibitem[{Nak}04]{Nak04}
Noboru {Nakayama}, \emph{{Zariski-decomposition and abundance}}, vol.~14,
  Tokyo: Mathematical Society of Japan, 2004 (English).

\bibitem[{Ogu}14]{Og14}
Keiji {Oguiso}, \emph{{Automorphism groups of Calabi-Yau manifolds of Picard
  number 2}}, {J. Algebr. Geom.} \textbf{23} (2014), no.~4, 775--795 (English).

\bibitem[{OSC}21]{oscar}
The {OSCAR team}, \emph{The {OSCAR} project},
  \url{https://oscar.computeralgebra.de/}, 2021.

\bibitem[SW90]{SW90}
Micha{\l} {Szurek} and Jaros{\l}aw~A. {Wi\'sniewski}, \emph{{Fano bundles over
  \(P^ 3\) and \(Q_ 3\)}}, {Pac. J. Math.} \textbf{141} (1990), no.~1, 197--208
  (English).

\bibitem[{Wis}89]{Wis89}
Jaroslaw~A. {Wisniewski}, \emph{{Ruled Fano 4-folds of index 2}}, {Proc. Am.
  Math. Soc.} \textbf{105} (1989), no.~1, 55--61 (English).

\bibitem[Yá21]{Y21}
José~Ignacio Yá{\~{n}}ez, \emph{Birational automorphism groups and the
  movable cone theorem for {C}alabi-{Y}au complete intersections of products of
  projective spaces}, preprint: \url{https://arxiv.org/pdf/2103.11638.pdf},
  2021.

\end{thebibliography}

\end{document}